\documentclass[12pt]{article}

\usepackage{amsmath}    
\usepackage{amssymb}    
\usepackage{amsthm}			

\usepackage{esint} 			

\usepackage{hyperref} 	

\theoremstyle{thmit} 
\newtheorem{thm}{Theorem}[section]
\newtheorem{lem}[thm]{Lemma}

\theoremstyle{thmrm} 

\newtheorem*{conjecture}{Conjecture}
\newtheorem*{oldproof}{Proof}
\renewenvironment{proof}[1][{}]{\begin{oldproof}[#1]}{\qed\end{oldproof}}

\def\sR{\mathcal{R}}
\def\OO{\widetilde{O}}

\def\nicebreak{\vskip 0pt plus 50pt\penalty-300\vskip 0pt plus -50pt } 

\def\e{\ell}
\def\nts{\negthickspace}

\def\c{\theta}
\def\d{\delta}

\def\ep{\varepsilon}
\def\l{\lambda}

\def\Z{\mathbb{Z}}
\def\R{\mathbb{R}}

\def\sG{\mathcal{G}}

\def\M{\mathcal{M}}

\def\md#1{\hat{#1}}
\def\mdA{\md{A}}
\def\mdB{\md{B}}
\def\mdC{\md{C}}
\def\mdD{\md{D}}
\def\mdE{\md{E}}
\def\mdF{\md{F}}
\def\mdG{\md{G}}
\def\mdH{\md{H}}
\def\mdI{\md{I}}
\def\mdZ{\md{Z}}

\def\Im{\operatorname{Im}}

\def\zvec{\boldsymbol{z}}

\def\cvec{\boldsymbol{\theta}}

\def\sumjk{\sum_{1\le j<k\le n}}
\def\sumj{\sum_{j=1}^{n}}
\def\sumk{\sum_{k=1}^{n}}
\def\prodjk{\prod_{1\le j<k\le n}}
\def\prodj{\prod_{j=1}^{n}}

\def\({\bigl(}
\def\){\bigl)}
\def\sRc{{\sR^c}}
\def\sRo{{\sR_0}}
\def\sRpi{{\sR_\pi}}
\def\sRz{{\sR_{\zvec}}}
\def\sP{\mathcal{P}}
\def\sS{\mathcal{S}}
\def\sG{\mathcal{G}}
\def\Prob{\operatorname{Prob}}

\providecommand{\abs}[1]{\lvert#1\rvert} 
\def\dfrac#1#2{\lower0.15ex\hbox{\large$\frac{#1}{#2}$}} 

\begin{document}

\title{Asymptotic enumeration of symmetric integer matrices with
uniform row sums}

\author	{Brendan~D.~McKay\\[5pt]
                                        \normalsize Research School of Computer Science\\[-3pt]
                                        \normalsize Australian National University\\[-3pt]
                                        \normalsize Canberra ACT 0200, Australia\\[5pt]
                                        \normalsize \ttfamily{bdm@cs.anu.edu.au}\\ \\
				Jeanette~C.~McLeod\\[5pt]
                                        \normalsize Department of Mathematics and Statistics\\[-3pt]
                                        \normalsize University of Canterbury\\[-3pt]
                                        \normalsize Christchurch 8140, New Zealand\\[5pt]
                                        \normalsize \ttfamily{jeanette.mcleod@canterbury.ac.nz}
}
	
\date{} 

\maketitle

\begin{abstract}
  We investigate the number of symmetric matrices of non-negative
  integers with zero diagonal such that each row sum is the same. Equivalently,
  these are zero-diagonal symmetric contingency tables with uniform
  margins, or loop-free regular multigraphs. We determine the
  asymptotic value of this number as the size of the matrix tends to
  infinity, provided the row sum is large enough. We conjecture that
  one form of our answer is valid for all row sums.
 \end{abstract}

\section{Introduction}\label{S:intro}

Let $M(n,\e)$ be the number of $n\times n$ symmetric matrices over $\{0,1,2,\dots\}$ with zeros on the main diagonal and each row summing to $\e$.
Our interest is in the asymptotic value of $M(n,\ell)$ as $n\to\infty$ with $\ell$ being a function of~$n$.
Alternative descriptions of the class $M(n,\e)$ are: adjacency matrices of loop-free regular multigraphs of order $n$ and degree~$\e$, and zero-diagonal symmetric contingency tables of dimension $n$ with uniform margins equal to~$\e$. An example appears in Figure~\ref{F:example}.

\begin{figure}[t] \label{F:example}
\parbox{0.4\textwidth}{
\begin{center}
	$n=9$\\[1em]
	$\hspace{0.6em} \e=20$\\[1em]
	$\hspace{0.6em} \l=\frac{\ell}{n-1}=2.5$
\end{center}
}%
\parbox{0.6\textwidth}{
		\renewcommand{\arraystretch}{1}
		\renewcommand{\tabcolsep}{0.3em}
		\begin{tabular}{ccccccccc}
			0&4&1&3&2&2&3&1&4\\  
			4&0&4&1&3&2&2&3&1\\
			1&4&0&7&1&0&2&2&3\\
			3&1&7&0&1&1&3&2&2\\
			2&3&1&1&0&7&1&3&2\\
			2&2&0&1&7&0&4&1&3\\
			3&2&2&3&1&4&0&4&1\\
			1&3&2&2&3&1&4&0&4\\
			4&1&3&2&2&3&1&4&0\\
		\end{tabular} 
}
\caption{An example of a matrix counted by $M(9,20)=1955487489759152410696$.}
\end{figure}

Very little seems to be known about this problem. 
The asymptotic value of $M(n,3)$ was determined by Read in 1958~\cite{Read}.
According to Bender and Canfield~\cite{BC1978}, de Bruijn extended this to
$M(n,\ell)$ for fixed $\ell$ but failed to publish it.
In any case,~\cite{BC1978} generalised the result to bounded but possibly
non-equal row sums. 
By the method of switchings, Greenhill and McKay~\cite{Sissy} found the asymptotic
number of matrices with given small row sums over a range that includes
$M(n,\ell)$ for $\ell=o(n^{1/2})$.

In this paper we treat the case of large~$\ell$ and manage to find the asymptotics whenever
$\ell>C n/\log n$ for any $C>\tfrac16$.
We will use the multidimensional saddle-point method, which was previously applied successfully to 
the corresponding $\{0,1\}$ problem by McKay and Wormald~\cite{MWreg}
and to the corresponding non-symmetric problem by Canfield and McKay~\cite{CMint}.
For the non-symmetric problem with mixed row and column sums, see
Barvinok and Hartigan~\cite{Barvinok}.

Our theorem is as follows.

\begin{thm}\label{Th:main}
Let $a$ and $b$ be positive real numbers such that $a+b<\tfrac{1}{2}$. Let $\e=\e(n)$ be such that $\e n$ is even and $\lambda=\ell/(n-1)$ satisfies
\begin{equation}\label{E:constant}
	\l\geq \frac{1}{3a\log n}.
\end{equation}
Then as $n\to \infty$, 
\begin{align}
		M(n,\ell) &= \sqrt{2}\,\(2\pi n(1+\l)^{-\e-n+2}\l^{\e+1}\)^{-n/2}
				\exp\biggl(\frac{14\l^2+14\l-1}{12\l(1+\l)}+O(n^{-b})\!\biggr)\notag\\
			&= \biggl(\frac{\lambda^\lambda}{(1+\lambda)^{1+\lambda}}
			            \biggr)^{\!\!\binom{n}{2}}
			       \binom{n+\ell-2}{\ell}^{\!n}
			      \sqrt 2\, e^{3/4}\, \(1+O(n^{-b})\).\label{E:binform}
\end{align}
\end{thm}

In Section~\ref{S:integral}, we express $M(n,\ell)$ as an integral
in $n$-dimensional complex space and divide the domain of
integration into three parts, then in Section~\ref{S:eval_integral} we
estimate the integral in two of the parts.
In Section~\ref{S:box}, we show that the third part is negligible in
comparison provided $\ell$ is bounded by a polynomial in~$n$.
We complete the proof for large $\ell$ in Section~\ref{S:proof}
using the theory of Ehrhart quasipolynomials.

In Section~\ref{S:naive}, we show that the form of expression~\eqref{E:binform} is motivated by a na\"\i ve probabilistic model.
We also note that~\eqref{E:binform} agrees with \cite{Sissy}, apart from the error term, when
$1\le\e=o(n ^{1/2})$, and closely matches many exact values computed as described in Section~\ref{S:exact}.  This leads us to suspect that~\eqref{E:binform} is true whenever $\e>0$, and we conjecture explicit bounds for $M(n,\e)$ in Conjecture~\ref{Conj:bold}.

Throughout the paper, asymptotic notation like $O(f(n))$
refers to the passage of $n$ to~$\infty$.
We will also use a modified notation $\OO(f(n))$.
A function $g(n)$ belongs to this class provided that 
\begin{equation*}
	g(n) = O(f(n)n^{a\ep}),
\end{equation*}
for some numerical constant $a$ that might be different at each use of the notation.

    

\section{An integral for $M(n,\e)$}
\label{S:integral}

We now express $M(n,\e)$ as an integral in $n$-dimensional complex space
and outline a plan for estimating it.

We begin with a generating function in $n$ variables $x_1,\ldots,x_n$,
\begin{equation*}
        \prodjk (1 - x_jx_k)^{-1},
\end{equation*}
for which the coefficient of $x_1^{\e_1}\cdots x_n^{\e_n}$ is the
number of $n\times n$ symmetric matrices over $\{0,1,2,\dots\}$ with
zeros on the main diagonal and row sums $\e_1,\ldots,\e_n$.
In particular, $M(n,\e)$ is the coefficient of $x_1^\e\cdots x_n^\e$. 

Applying Cauchy's Integral Formula we have

\begin{equation}\label{cauchy1}
 M(n,\e) = \frac{1}{(2\pi i)^n} \oiint
		\frac{\prodjk(1-x_j x_k)^{-1}}{x_1^{\e+1}\cdots x_n^{\e+1}}\, dx_1\cdots dx_n,
\end{equation}
where each variable is integrated along a contour circling the origin once in the anticlockwise direction.
It will suffice to take the contours to be circles; specifically, we will put $x_j=re^{i\c_j}$ for each $j$, where, for reasons that will become clear in Section~\ref{S:eval_integral},
we choose
\begin{equation*}
	r = \sqrt{\frac{\l}{1+\l}}\,.
\end{equation*}

This gives
\begin{equation*}
	M(n,\e) = \frac{1}{(2\pi)^n} \(\l^{-\l}(1+\l)^{1+\l}\)^{\binom{n}{2}}\, I(n),
\end{equation*}
where
\begin{equation}\label{E:In}
	I(n) =	\int_{-\pi}^{\pi}\nts\cdots \int_{-\pi}^{\pi}
 			\frac{\prodjk \(1 - \l(e^{i(\c_j+\c_k)}-1)\)^{-1}}{e^{il\sumj\c_j}}
   			\, d\cvec.
\end{equation}
Let $F(\cvec)$ be the integrand in~\eqref{E:In}. 

The quantity $(1 - \l(e^{i(\c_j+\c_k)}-1))^{-1}$, and thus $F(\cvec)$,
has greatest magnitude when $\c_j+\c_k\in\{0,2\pi\}$ for
each distinct pair~$j,k$. It is easy to see that these constraints have only two solutions: $\c_j=0$ for all~$j$, and $\c_j=\pi$ for all~$j$.
We will show that the value of $I(n)$ comes mostly from the neighbourhoods
of these two points; specifically, it comes from two boxes
$\sRo,\sRpi\subseteq[-\pi,\pi]^n$ defined, for sufficiently small $\ep$,
as 
\begin{align*}
	\sRo &= \{\boldsymbol{\c} : \abs{\c_j} \leq n^{-1/2 + \ep}(1+\l)^{-1}
	  \text{~for all $j$}\, \}, \text{~and}
\\
	\sRpi &= \{\boldsymbol{\c} : \abs{\c_j+\pi} \leq n^{-1/2 + \ep}(1+\l)^{-1} 
	  \text{~for all $j$}\, \},
\end{align*} 
where $\c_j+\pi$ is taken mod $2\pi$. Note that the operation $\c_j\!\mapsto \c_j+\pi$ for all~$j$, which maps $\sRo$ and $\sRpi$ onto each other, preserves $F(\cvec)$ since $n\e$ is even.
Also note that $\sRo\cap\sRpi=\emptyset$. We denote the region outside of the boxes as 
\begin{equation} \label{E:Rc}
	\sRc=[-\pi,\pi]^n\setminus(\sRo\cup\sRpi).
\end{equation}

If $X\subseteq[-\pi,\pi]^n$, then we let $I_X(n)=\int_X F(\cvec)\,d\cvec$. For $\l=O(n^5)$ we will evaluate the integral $I(n)$ defined in~\eqref{E:In} in the following way:
\begin{align}
	I(n) 	& = I_{\sRo}(n) + I_{\sRpi}(n) + I_{\sRc}(n)\notag \\ 
				& = 2I_{\sRo}(n) + O(1)\int_{\sRc}\!\abs{F(\cvec)}d\cvec\notag \\ 
				& = 2I_{\sR'}(n) + O(1)\int_{\sRc}\!\abs{F(\cvec)}d\cvec \label{E:In_eval}
\end{align}
for any $\sR'$ with $\sRo\subseteq\sR'\subseteq[-\pi,\pi]^n\setminus\sRpi$.



\nicebreak
\section{The main part of the integral}\label{S:eval_integral}

In this section we estimate the value of the integral $I(n)$ in a
convenient region~$\sR'$ that contains~$\sRo$. We begin by quoting several results required for the calculation. 

The following theorem, simplified from~\cite{RANX}, estimates the value of a certain multidimensional integral.

\begin{thm}\label{Th:multi_dim_int}
Let $\ep', \ep'', \ep''', \check\ep$ be constants such that $0<\ep'<\ep''<\ep'''$, and $\check\ep > 0$.  The following is true if $\ep'''$ is sufficiently small.

Let $\mdA=\mdA(n)$ be a real-valued function such that $\mdA(n)=\Omega(n^{-\ep'})$.
Let $\mdB=\mdB(n)$, $\mdC=\mdC(n)$, $\mdE=\mdE(n)$, $\mdF=\mdF(n)$, $\mdG=\mdG(n)$, $\mdH=\mdH(n)$, and $\mdI=\mdI(n)$ be complex-valued functions of $n$ such that $\mdB, \mdC, \mdE, \mdF, \mdG, \mdH, \mdI=O(1)$.
Suppose $\hat\ep(n)$ satisfies $\ep''\le2\hat\ep(n)\le\ep'''$ for all~$n$ and define
\begin{equation*}
	U_n = \bigl\{\zvec \subseteq \R^n : \abs{z_j} \le n^{-1/2+\hat\ep(n)} 
			\text{ for\/ $1\le j\le n$}\bigr\}.
\end{equation*}
Suppose that for $\zvec=(z_1,z_2,\ldots,z_n)\in U_n$ we have
\begin{align*}
	f(\zvec)= 	& \exp\biggl(-\mdA n\sumj z_j^2 + \mdB n \sumj z_j^3 
								+ \mdC\!\sum_{j,k=1}^n z_j z_k^2 
								+ \mdD n^{-1}\nts\sum_{j,k,p=1}^n\! z_j z_k z_p\\
		&{\qquad} + \mdE n \sumj z_j^4
							  + \mdF\sum_{j,k=1}^n z_j^2 z_k^2 + \mdG n^{1/2} \sum_{j,k=1}^n z_j z_k^3\\
		&{\qquad} + \mdH n^{-1/2}\nts\sum_{j,k,p=1}^n z_j z_k z_p^2
								+ \mdI n^{-3/2}\nts\nts\sum_{j,k,p,q=1}^n\nts\nts z_j z_k z_p z_q
							 	+ \d(\zvec) \biggr),
\end{align*}
where $\delta(\zvec)$ is continuous and $\delta(n)=\max_{\zvec\in U_n} \abs{\delta(\zvec)} = o(1)$.
Then, provided the $O(\,)$ term in the following converges to zero,
\begin{equation*}
	\int_{U_n}f(\zvec)\,d\zvec 
								= \biggl(\frac{\pi}{\mdA n}\biggr)^{\!n/2}\!\exp\Bigl(\varTheta_1 
								+ O\bigl(n^{-1/2+\check\ep}+(n^{-3/4}+\d(n))\mdZ\bigr)\Bigr),
\end{equation*}
where
\begin{align*}
	\varTheta_1 	& = \frac{15\mdB^2}{16\mdA^3} + \frac{3\mdB\mdC}{8\mdA^3} 
					+ \frac{\mdC^2}{16\mdA^3} + \frac{3\mdE}{4\mdA^2} + \frac{\mdF}{4\mdA^2},\; \text{ and}\\[10pt]
	\mdZ 		& = \exp\biggl(\frac{15\Im(\mdB)^2 + 6\Im(\mdB)\Im(\mdC) 
					+ \Im(\mdC)^2}{16\mdA^3} \biggr).
\end{align*}

\end{thm}

The following lemma defines a linear transformation, adapted from~\cite{MWreg}. 

\begin{lem}\label{L:diagonal_transform}
Define $c$ and $\zvec=(z_1,z_2,\ldots,z_n)$ by
\begin{align}
   c &= 1 - \sqrt{\frac{n-2}{2(n-1)}} = 1-2^{-1/2}+O(n^{-1}), \label{E:T1a}\\
   (1+ \l)\,\theta_j &= z_j - \frac{c}{n}\sumk z_k\qquad(1\le j\le n).\label{E:T1b}
\end{align}
The transformation $\cvec=T(\zvec)$ defined by~\eqref{E:T1b} has determinant~$(1-c)/(1+\l)^n$. For $m\ge 1$, define $\mu_m=\sumj z_j^m$.
Then we have the following translations.
\begin{align*}
  (1+ \l)\sumj\theta_j																
  	&= (1-c)\mu_1, \\[0pt]
  (1+ \l)^2\nts\nts\sumjk\nts (\theta_j+\theta_k)^2 	
  	&= (n-2)\mu_2, \displaybreak[0] \\[4pt]
  (1+ \l)^3\nts\nts\smash\sumjk\nts\nts (\theta_j+\theta_k)^3 	
  	&= (n-4)\mu_3 + \bigl(3(1-2c) + 12c/n\bigr)\mu_1\mu_2 \\[0pt]
   	&\quad{}+ \bigl((-6c+12c^2-4c^3)/n - 4c^2(3-c)/n^2\bigr)\mu_1^3, \\[4pt]
  (1+ \l)^4\nts\nts\smash\sumjk\nts (\theta_j+\theta_k)^4 	
  	&= (n-8)\mu_4 + 3\mu_2^2 + \bigl(4(1-2c)+32c/n\bigr)\mu_1\mu_3 \\[0pt]
  	&\quad{} -\bigl(24c(1-c)/n+48c^2/n^2\bigr)\mu_1^2\mu_2 \\[0pt]
		&\quad{} +\bigl( 8c^2(1-c)(3-c)/n^2+8c^3(4-c)/n^3\bigr)\mu_1^4.
\end{align*}
\end{lem}

{}From Taylor's Theorem with remainder we have

\begin{lem}\label{L:int_prod_ub}
For all real $X$,
\begin{align*}
	\bigl(1-\l(e^{iX}-1)\bigr)^{-1}	= 	& \exp\bigl(\l iX - \tfrac{1}{2}\l(1+\l)X^2 
						- \tfrac{1}{6}i\l(1+\l)(1+2\l)X^3 \\
						& + \tfrac{1}{24}\l(1+\l)(1+6\l+6\l^2)X^4 + O((\l+\l^5)X^5)\bigr).
\end{align*}

\end{lem}

We now present the main result of this section.

\begin{thm} \label{Th:IR'}
Under the conditions of Theorem~\ref{Th:main}, there is a region $\sR'$ such that \text{$\sRo\subseteq\sR'\subseteq 3\sRo\subseteq[-\pi,\pi]^n\setminus\sRpi$} and
\begin{equation*}
		I_{\sR'}(n) = \frac{1}{\sqrt{2}}\biggl(\frac{2\pi}{\l(1+\l)n}\biggr)^{\nts n/2}\nts\exp\biggl(\frac{14\l^2+14\l-1}{12\l(\l+1)}+O(n^{-b})\biggr).
\end{equation*}
\end{thm}

\begin{proof}
Consider the transformation $\cvec=T(\zvec)$ defined by~\eqref{E:T1b}.
Define
\begin{equation*}
    \sR_{\zvec} = \{\zvec : \abs{z_j} \leq 2 n^{-1/2+\ep} \} \text{~and~}
    \sR' = T(\sRz).
\end{equation*}
{}From~\eqref{E:T1b} we have
\begin{align*}
 \abs{\c_j} &\leq y \text{ for all $j$} \implies \abs{z_j} \leq (1+\l)(1-c)^{-1}y \text{ for all $j$}, \\[6pt]
 \abs{z_j} &\leq y \text{ for all $j$} \implies \abs{\c_j} \leq (1+\l)^{-1}(1+c)y \text{ for all $j$}.
\end{align*}
These imply, for $n\ge 2$, that $T^{-1}\sRo\subseteq \sRz$
and
\begin{equation*}
	\sRo\subseteq\sR'\subseteq 3\sRo.
\end{equation*}
{}From Lemma~\ref{L:int_prod_ub} we have, for $\cvec\in \sR'$,
\begin{multline*}
	F(\cvec) 	= \exp\biggl(-A_2\nts\nts\nts\sumjk\nts(\c_j+\c_k)^2 
						- iA_3\nts\nts\nts\sumjk\nts(\c_j+\c_k)^3 \\
						+ A_4\nts\nts\nts\sumjk\nts(\c_j+\c_k)^4	
						+ \OO(n^{-1/2})\biggr),
\end{multline*}
where $A_2=\tfrac{1}{2}\l(1+\l)$, $A_3=\tfrac{1}{6}\l(1+\l)(1+2\l)$, $A_4=\tfrac{1}{24}\l(1+\l)(1+6\l+6\l^2)$.
The absence of a linear term is due to our particular choice of $r$ in
Section~\ref{S:integral}.


Using Lemma~\ref{L:diagonal_transform}, we perform the transformation $\cvec=T(\zvec)$. This diagonalizes the quadratic form in $F(\cvec)$, and $I_{\sR'}$ becomes:
\begin{equation*}
	\frac{1}{\sqrt{2}}\biggl(\frac{2\pi}{\l(1+\l)n}\biggr)^{\nts n/2}\nts
	\int_{\sRz}\nts F\bigl(T(\zvec)\bigr) \, d \zvec,
\end{equation*}
where
\begin{align*}
	F\bigl(T(\zvec)\bigr) =	
			&  \exp\(- A_2B_2(1+\l)^{-2}\mu_2 
				-i A_3B_3(1+\l)^{-3}\mu_3 \\
			&{\qquad}	-i A_3B_{1,2}(1+\l)^{-3}\mu_1\mu_2
				-i A_3B_{1,1,1}(1+\l)^{-3}\mu_1^3 \\
			&{\qquad}	 + A_4B_4(1+\l)^{-4}\mu_4 + A_4B_{2,2}(1+\l)^{-4}\mu_2^2\displaybreak[0]\\
    	& {\qquad}+	A_4B_{1,3}(1+\l)^{-4}\mu_1\mu_3
    	  - A_4B_{1,1,2}(1+\l)^{-4}\mu_1^2\mu_2\\
    	& {\qquad}+ A_4B_{1,1,1,1}(1+\l)^{-4}\mu_1^4
    	  + \OO(n^{-1/2})\),
\end{align*}
where
\begin{align*}
	B_2 				& = n-2,\\[0pt]
	B_3  				& = n-4,\\[0pt]
	B_4 				& = n-8,\\
	B_{1,2} 		& = 3(1-2c)+\frac{12c}{n} = -3 + 3\sqrt{2} + O(n^{-1}),\displaybreak[0]\\[2pt]
	B_{1,3} 		& = 4(1-2c)+\frac{32c}{n} = -4 + 4\sqrt{2} + O(n^{-1}),\\[0pt]
	B_{2,2} 		& = 3,\\[0.5em]
	B_{1,1,1}		& = \frac{-6c+12c^2-4c^3}{n} - \frac{4c^2(3-c)}{n^2} = \frac{2-2\sqrt{2}}{n}+ O(n^{-2}),\\[0pt]
	B_{1,1,2} 	& = -\frac{24c(1-c)}{n}+\frac{48c^2}{n^2} = O(n^{-1}),\\[0pt]
	B_{1,1,1,1}	& = \frac{8c^2(1-c)(3-c)}{n^2} + \frac{8c^3(4-c)}{n^3} = O(n^{-2}).
\end{align*}

In order to apply Theorem~\ref{Th:multi_dim_int} we choose $\hat\ep(n) = \ep + \log 2/\log n$, $\ep'=\tfrac{1}{2}\,\ep$, $\ep''=\ep$, $\ep'''=3\ep$, $\check\ep=\ep$, $\d(n)=\OO(n^{-1/2})$ and
\begin{align*}
	\mdA &= \frac{A_2B_2}{(1+\l)^2n} = -\frac{\l}{2(1+\l)}\biggl(1-\frac{2}{n}\biggr),\\[0pt]
	\mdB &= -i\,\frac{A_3B_3}{(1+\l)^3n} = -i\,\frac{\l(1+2\l)}{6(1+\l)^2} + O(n^{-1}),
	\displaybreak[0]\\[0pt]
	\mdC &= -i\,\frac{A_3B_{1,2}}{(1+\l)^3} 
	      = i\frac{\l(1+2\l)(1-\sqrt{2})}{2(1+\l)^2} + O(n^{-1}),\displaybreak[0]\\[0pt]
	\mdD &= -i\,\frac{A_3B_{1,1,1}n}{(1+\l)^3}
	      = -i\,\frac{\l(1+2\l)(1-\sqrt{2})}{3(1+\l)^2} + O(n^{-1}),\displaybreak[0]\\[0pt]
	\mdE &= \frac{A_4B_4}{(1+\l)^4n}
	      = \frac{\l(1+6\l+6\l^2)}{24(1+\l)^3} + O(n^{-1}),\displaybreak[0]\\[0pt]
	\mdF &= \frac{A_4B_{2,2}}{(1+\l)^4}
	      = \frac{\l(1+6\l+6\l^2)}{8(1+\l)^3},\displaybreak[0]\displaybreak[0]\\[0pt]
	\mdG &= \frac{A_4B_{1,3}}{(1+\l)^4n^{1/2}}
	      = O(n^{-1/2}),\\[0pt]
	\mdH &= \frac{A_4B_{1,1,2}n^{1/2}}{(1+\l)^4} 
	      = O(n^{-1/2}),\displaybreak[0]\\[0pt]
	\mdI &= \frac{A_4B_{1,1,1,1}n^{3/2}}{(1+\l)^4} 
	      = O(n^{-1/2}),\displaybreak[0] \\[4pt]
	\mdZ &= \exp\biggl(\frac{nA_3^2(15B_3^2+6B_3B_{1,2}n	
			  + B_{1,2}^2n^2)}{16B_2^3A_2^3}\biggr)\\
			 &= \exp\biggl(\frac{(1+2\l)^2}{3\l(1+\l)}
				+ \OO(n^{-1})\biggr),\\[0pt]
	\varTheta_1 &= \frac{2\l^2+2\l-1}{12\l(1+\l)} + \OO(n^{-1}).
\end{align*}

Theorem~\ref{Th:IR'} now follows from Theorem~\ref{Th:multi_dim_int}.
\end{proof}

\nicebreak
\section{Concentration of the integral}\label{S:box}

In the previous section we proved that the contribution to $I(n)$ from the box $\sR'$ is
\begin{equation*}
   I_{\sR'}(n) = \biggl(\frac{\pi}{A_2n}\biggr)^{\!n} \exp\(O(1+\lambda^{-1})\).
\end{equation*}
We now consider the contribution to $I(n)$ from the region $\sRc$ (defined in~\eqref{E:Rc}) and show, provided $\l$ is not too large, that it is negligible
compared to $I_{\sR'}(n)$.

First we import from~\cite{CMint} some useful lemmas.

\begin{lem}\label{L:fbnd}
The absolute value of the integrand $F(\cvec)$ of \eqref{E:In} is
\begin{equation*}
  \abs{F(\cvec)} = \smash{\prodjk}  f(\theta_j + \theta_k),
\end{equation*}
where
\begin{equation*}
  f(z) = \(1+4A_2(1-\cos z)\)^{-1/2}.
\end{equation*}
Moreover, for all real $z$ with $\abs z\le \tfrac1{10}(1+\lambda)^{-1}$,
\begin{equation}\label{E:fbound}
  0\le f(z) \le \exp\( -A_2 z^2 + (\tfrac1{12} A_2+A_2^2) z^4\).
\end{equation}
\end{lem}

\begin{lem}\label{L:ibnd}
Define $t=\tfrac{1}{60}(1+\lambda)^{-1}$
and $g(x)=-A_2x^2+(\frac34A_2+9A_2^2)x^4$.
Then, uniformly for $\lambda>0$ and $K\ge 1$,
\begin{equation*}
  \int_{-2t}^{2t} \exp\( K g(x) \)\,dx
 \le \sqrt{\pi/(A_2K)}\, \exp\( O(K^{-1}+(A_2K)^{-1})\).
\end{equation*}
\end{lem}

\begin{thm}\label{Th:boxing}
Suppose that the conditions of Theorem~\ref{Th:main} hold,
and in addition that $\lambda=n^{O(1)}$.
Then
\begin{equation*}
      \int_\sRc \abs{F(\cvec)}\,d\cvec = O(n^{-1})I_{\sR'}(n).
\end{equation*}
\end{thm}

\begin{proof}
The proof follows a similar pattern to that of~\cite[Theorem~1]{MWreg}.
Define $t$ and $g(z)$ as in Lemma~\ref{L:ibnd}.

Define $n_0,n_1,n_2,n_3$, functions of $\cvec$, to be the number of
indices $j$ such that $\theta_j$ lies in $[-t,t]$, $(t,\pi-t)$,
$[\pi-t,\pi+t]$, and $(-\pi+t,-t)$, respectively.
Let $\sR''$ be the set of all $\cvec$ such that 
$\max\{n_0n_2,\binom{n_1}{2},\binom{n_3}{2}\}\ge n^{1+\ep}$.
Any $\cvec\in\sR''$ has the property that
$f(\theta_j+\theta_k)\le f(2t)$ for at least $n^{1+\ep}$
pairs $j,k$.  Since $f(z)\le 1$ for all~$z$, and the volume
of $\sR''$ is less than $(2\pi)^n$, we have
\begin{equation*}
   \int_{\sR''} \abs{F(\cvec)}\,d\cvec \le (2\pi)^n f(2t)^{n^{1+\ep}}.
\end{equation*}
Applying~\eqref{E:fbound} and the assumption that $\lambda=O(n^{O(1)})$,
we find that 
\begin{equation}\label{E:part1}
   \int_{\sR''} \abs{F(\cvec)}\,d\cvec = O(e^{-c_1 n^{1+\ep/2}}) I_{\sR'}(n)
\end{equation}
for some $c_1>0$.

For $\cvec\in\sRc\setminus \sR''$
we must have $n_1,n_3 = O(n^{1/2+\ep})$ and either
$n_0=O(n^{1/2+\ep})$ or $n_2=O(n^{1/2+\ep})$.
The latter two cases are equivalent, so we will assume that
$n_2=O(n^{1/2+\ep})$, which implies that $n_0=n-O(n^{1/2+\ep})$.

Define $S_0,S_1,S_2$, functions of $\cvec$, as follows.
\begin{align*}
  S_0 &= \{\, j : \abs{\theta_j} \le t\,\}, \\
  S_1 &= \{\, j : t < \abs{\theta_j} \le 2t\,\}, \\
  S_2 &= \{\, j : \abs{\theta_j} > 2t\,\}.
\end{align*}
Define $s_i=\abs{S_i}$ for each~$i$.
Since $s_0 = n_0$, we know that $s_1+s_2=O(n^{1/2+\ep})$.
Now we bound $\abs{F(\cvec)}$ in $\sRc\setminus \sR''$ using
\begin{equation*}
  f(\theta_j+\theta_k) \le \begin{cases}
        \,f(t) \le \exp\biggl(\displaystyle
                   -\frac{\lambda}{14400(1+\lambda)}\biggr)
                         & \hspace{-0.7em} \text{if $j{\in}S_0, k{\in}S_2$}, \\[0.4ex]
        \,\exp\(-A_2(\theta_j+\theta_k)^2
                 +(\tfrac1{12}A_2+A_2^2)(\theta_j+\theta_k)^4\)
                         & \hspace{-0.7em} \text{if $j,k{\in}S_0$}, \\[0.4ex]
        \,1  & \hspace{-0.7em} \text{otherwise}. \end{cases}
\end{equation*}
Let $I_2(s_2)$ be the contribution to $I(n)$ from those
$\cvec\in\sRc\setminus \sR''$ with the given value of~$s_2$,
and let $\cvec'$ denote the vector $(\theta_j)_{j\in S_0}$.
The set $S_2$ can be chosen in at most $n^{s_2}$ ways.
Applying the bounds above, and allowing
$(2\pi)^{s_1+s_2}$ for integration over $\theta_j\in S_1\cup S_2$,
we find
\begin{equation}\label{E:I2s2}
 I(s_2) \le n^{s_2} (2\pi)^{s_1+s_2}
             \exp\biggl(-\frac{s_0s_2\lambda}{14400(1+\lambda)}\biggr) I'(s_0),
\end{equation}
where
\begin{align*}
  I'(s_0) &= \int_{-t}^{t}\!\!\cdots \int_{-t}^{t}\,
         		\prod_{j,k\in S_0, j<k}\nts\nts f(\theta_j+\theta_k)\,d\cvec' \\
     			&\le \int_{-t}^{t}\!\!\cdots \int_{-t}^{t}\,
						\exp\Bigl(-A_2\nts\nts\sum_{j,k\in S_0,j<k}\nts\nts(\theta_j+\theta_k)^2 \\
        	& \qquad \qquad \qquad \qquad +(\tfrac1{12}A_2+A_2^2)\nts\nts
        		\sum_{j,k\in S_0,j<k}\nts\nts(\theta_j+\theta_k)^4\Bigr)\,d\cvec'
           	\displaybreak[0]\\
     			&\le \int_{-t}^{t}\!\!\cdots \int_{-t}^{t}\,
        		\exp\Bigl(-A_2(s_0-2)\sum_{j\in S_0} \theta_j^2 \\
          & \qquad \qquad \qquad \qquad	+ 8(s_0-1)(\tfrac1{12}A_2+A_2^2)\sum_{j\in S_0} \theta_j^4\Bigr)\,d\cvec'
           	\displaybreak[0]\\
     			&\le \biggl(\, \int_{-t}^{t} 
	\exp\(-(s_0-2)g(z)\)\,dz
	  \biggr)^{\nts s_0},  \quad 
	             \text{ for } s_0 \ge 10, \displaybreak[0]\\
     &\le \biggl( \sqrt{\frac{\pi}{A_2(s_0{-}2)}}
	   \exp\( O(1+\lambda^{-1})n^{-1} \) \biggr)^{\nts s_0} \\
     &\le \biggl( \frac{\pi}{A_2n} \biggr)^{\!n/2} 
	   \exp\(O(n^{1/2+2\ep})\).
\end{align*}
The third line of the above follows from the bounds
\begin{equation*}
\sum_{1\le j<k\le p}(x_j+x_k)^2\ge (p-2)\sum_{j=1}^p x_j^2
\ \text{ and}
\sum_{1\le j<k\le p}(x_j+x_k)^4\le 8(p-1)\sum_{j=1}^p x_j^4
\end{equation*}
valid for all $x_1,x_2,\ldots,x_p$.
The fifth line follows from Lemma~\ref{L:fbnd}, and the last line
follows from $s_0=n-O(n^{1/2+\ep})$.
Substituting this bound into~\eqref{E:I2s2} we find that
\begin{equation}\label{E:part2}
   \sum_{s_2\ge 1} I_2(s_2) = O\(\exp(-c_2n/\log n)\) I_{\sR'}(n)
\end{equation}
for some $c_2>0$.

With the cases of~\eqref{E:part1} and~\eqref{E:part2} excluded, we are
left with the problem of bounding the contribution of
$\cvec\in[-2t,2t]^n\setminus\sR'$.
Let $u=n^{-1/2+\ep}/(1+\lambda)$.
First note that, by arguing as above, we have that
\begin{equation*}
  \abs{F(\cvec)} \le \prodj \exp\( (n-2) g(\theta_j)\)
\end{equation*}
for all $\cvec\in[-2t,2t]^n$.
Also note that, by Lemma~\ref{L:ibnd},
\begin{equation}\label{E:fullbit}
   \int_{-2t}^{2t} \exp\( (n-2) g(z)\)\,dz
       \le \sqrt{\pi/(A_2n)} \exp\( O(1+\lambda^{-1})n^{-1} \).
\end{equation}
The function $g(z)$ has at most one minimum in $[u,2t]$, and
$g(2t)<g(u)$ for sufficiently large~$n$, so
\begin{align}
   \biggl( \int_{-2t}^{-u}\!\!+\int_{u}^{2t}\biggr) 
	\exp\( (n-2) g(z)\)\,dz
   &\le 4t\exp\( (n-2) g(u) \) \notag\\
   &\le \exp\biggl( -\frac{\lambda}{4(1+\lambda)} n^\ep\biggr)
	\label{E:partbit}.
\end{align}
Let $J_1,J_2$ be the right sides of~\eqref{E:fullbit} and~\eqref{E:partbit},
respectively.  Then
\begin{align}
   \int_{[-2t,2t]^n\setminus\sR} \abs{F(\cvec)}\, d\cvec
		&\le \sum_{q=1}^n \binom{n}{q} J_2^q J_1^{n-q} \notag\\
		&= J_1^n \( (1+J_2/J_1)^n-1 \) \notag\\
    &= O\(e^{-c_3 n^{\ep}}\) I_{\sR'}(n) \label{E:part3}
\end{align}
for some $c_3>0$.

The lemma now follows from~\eqref{E:part1}, \eqref{E:part2}
and~\eqref{E:part3}.
\end{proof}

\nicebreak
\section{Proof of Theorem \ref{Th:main}}\label{S:proof}

In the case that $\lambda=n^{O(1)}$,
Theorem~\ref{Th:main} follows from Theorem~\ref{Th:IR'} and
Theorem~\ref{Th:boxing}.  For larger $\lambda$, the method used
in the proof of Theorem~\ref{Th:boxing} is insufficient so we
need a new approach.

Let us assume that we already proved~Theorem~\ref{Th:main} for
$\lambda=O(n^5)$.  Now we want to show that it must be true
for larger $\lambda$ as well.  First note that
for such large $\lambda$ (indeed for $\lambda/n\to\infty$),
Theorem~\ref{Th:main} is equivalent to
\begin{equation}\label{E:biglambda}
   M(n,\ell) = \sqrt 2\, \(\lambda+\tfrac12\)^{n(n-3)/2}\,
   \frac{e^{\binom{n}{2}+7/6}} {(2\pi n)^{n/2}}
    \( 1+O(n^2/\lambda^2+n^{-b})\).
\end{equation}

Let $\sP_n$ be the polytope of symmetric $n\times n$ real
non-negative matrices with zero diagonal whose rows sum to~1.
Then $M(n,\ell)$ is the number of integer points in $\ell\sP_n$.
That is, $M(n,\ell)$ is the \textit{Ehrhart quasipolynomial}
of~$\sP_n$.

According to \cite[Theorem 8.2.6]{BrualdiBook}, the vertices of $\sP_n$
are the adjacency matrices of graphs whose components are either isolated
edges of weight~1
or odd cycles with edges of weight~$\tfrac12$.
That is, the coordinates of the vertices are multiples of~$\tfrac12$.
By a result of Ehrhart (see ~\cite[Ex. 3.25]{BeckRobinsBook}),
there is a polynomial $f_n(z)$ with non-negative integer
coefficients such that
\begin{equation}\label{E:gf}
   \sum_{\ell\le 0} M(n,\ell) z^\ell
     = \frac{f_n(z)}{(1-z^2)^{d+1}},
\end{equation}
where $d$ is the dimension of $\sP_n$.
By~\cite{Stanley73}, $d=n(n{-}3)/2$.

By applying the binomial expansion to $(1-z^2)^{-d-1}$ in~\eqref{E:gf},
we find that $M(n,\ell)$ is a polynomial in $\ell$ for even $\ell$
and a possibly-different polynomial in $\ell$ for odd~$\ell$.
Explicitly, there are non-negative integers
$h_0,\ldots,h_d$ (dependent on~$n$ and the parity of $\ell$)
such that
\[
    M(n,\ell) = \sum_{i=0}^d h_{d-i} \binom{\ell+i}{d}.
\]
Arguing as in~\cite{CMint}, we infer that there is a
function $\alpha(n,\ell)$ such that
\begin{align}
   M(n,\ell)& =
   \binom{\ell}{d} \biggl(\sum_{i=0}^d h_i\biggr)
      \(1 + \alpha(n,\ell)/\ell\) \label{E:bit1}\\
      \alpha(n,\ell)&\ge 0
        \text{~~for $\ell\ge d$} \label{E:bit2}\\[0.3ex]
   \alpha(n,\ell+2)&\le \alpha(n,\ell)
        \text{~~for $\ell\ge d$}. \label{E:bit3}
\end{align}
Equations~\eqref{E:biglambda} and~\eqref{E:bit1} both apply for
$\ell=\varTheta(n^5)$, so for $m\in\{n^5,n^5+1\}$,
\[
   \frac{M(n,3m)}{M(n,m)}
     = 3^d \,\frac{1+\tfrac13\alpha(n,3m)/m^5}
                   {1+\alpha(n,m)/m^5}
            \,\(1 + O(n^{-1})\)
     = 3^d\(1 + O(n^{-b})\),
\]
where the first estimate comes from~\eqref{E:bit1} and the
second comes from~\eqref{E:biglambda}.
Comparing these two estimates, and noting from~\eqref{E:bit2}
and~\eqref{E:bit3}
that $0\le \alpha(n,3m)\le\alpha(n,m)$,
we conclude that $\alpha(n,m)=O(n^{5-b})$.
By~\eqref{E:bit3}, this implies that $\alpha(n,\ell)=O(n^{5-b})$
for all $\ell\ge n^5$.
Now we can see from~\eqref{E:bit1} that
\[
  M(n,\ell) = M(n,m)\frac{M(n,\ell)}{M(n,m)}
     =  M(n,m)\biggl(\frac{\ell}{m}\biggr)^{\!\!d}
      \(1 + O(n^{-b})\)
\]
and apply~\eqref{E:biglambda} to $M(n,m)$.
This shows that~\eqref{E:biglambda} holds for all $\ell\ge n^5$.
The proof of Theorem~\ref{Th:main} is now complete.

\nicebreak
\section{Na\"\i ve Thinking}\label{S:naive}

In this section we consider a ``na\"\i ve'' model of
random matrix and show how it motivates our estimate
for $M(n,\l)$.

Define $\sG_\lambda$ to be the geometric distribution with
mean~$\lambda$.  That is, for a random variable $X$ distributed
according to $\sG_\lambda$, we have
\begin{equation}\label{E:geometric}
    \mathrm{Prob}(X=j) = \frac{1}{1+\lambda}
        \biggl( \frac{\lambda}{1+\lambda} \biggr)^{\!j}
\end{equation}
for $j\ge 0$.

Define $\sS=\sS(n,\ell)$ to be the probability space of $n\times n$
non-negative symmetric integer matrices with zero diagonal,
where each element of the upper triangle is independently
chosen from $\sG_\lambda$.
Define events on $\sS$:
\begin{align*}
  E_j &: \text{row $j$ has sum $\ell$}\\
  E_{\mathrm{all}} &: \bigcap_{j=1}^n E_j\\
  E_0 &: \text{the whole matrix has sum $n\ell$.}
\end{align*}
Note that $E_{\mathrm{all}}\subseteq E_0$.
Also note that each matrix in $E_0$ has the same probability,
namely
\begin{equation*}
  P_0 = \biggl( \frac{1}{1+\lambda} \biggr)^{\!\binom{n}{2}}
        \biggl( \frac{\lambda}{1+\lambda} \biggr)^{\nts n\ell/2}.
\end{equation*}
(Proof: Apply~\eqref{E:geometric} to each entry in
the upper triangle and use the
assumed independence of the entries there. The result is
independent of the actual matrix entries.)
Therefore,
\begin{equation*}
  M(n,\ell) = \frac{\mathrm{Prob}(E_{\mathrm{all}})}{P_0}.
\end{equation*}

Now make a \textbf{na\"\i ve assumption} that the events $E_j$
are independent.  By symmetry, $\mathrm{Prob}(E_j)$ is
independent of~$j$, so
we get a na\"\i ve estimate of $M(n,\ell)$:
\begin{equation}\label{E:naive}
  M_{\mathrm{naive}}(n,\ell) = \frac{\mathrm{Prob}(E_1)^n}{P_0}.
\end{equation}
Now consider $\mathrm{Prob}(E_1)$. The number of possible
first rows is
\begin{equation*}
    \binom{n+\ell-2}{\ell}.
\end{equation*}
(This is the number of ways of writing $\ell$ as the sum
of $n-1$ non-negative integers.)
In space $\sS$, each such first row has probability
\begin{equation*}
  \biggl( \frac{1}{1+\lambda} \biggr)^{\!n-1}
        \biggl( \frac{\lambda}{1+\lambda} \biggr)^{\!\ell}.
\end{equation*}
Therefore,
\begin{equation*}
  \mathrm{Prob}(E_1) = \binom{n+\ell-2}{\ell}
    \biggl( \frac{1}{1+\lambda} \biggr)^{\!n-1}
        \biggl( \frac{\lambda}{1+\lambda} \biggr)^{\!\ell}.
\end{equation*}
Substituting this value into~\eqref{E:naive}, we get
\begin{equation*}
  M_{\mathrm{naive}}(n,\ell) =
  \biggl(\frac{\lambda^\lambda}{(1+\lambda)^{1+\lambda}}
            \biggr)^{\!\!\binom{n}{2}}
       \binom{n+\ell-2}{\ell}^{\!n}
\end{equation*}
Therefore, formula~\eqref{E:binform} in Theorem~\ref{Th:main} can be written
\begin{equation*}
   M(n,\ell) = M_{\mathrm{naive}}(n,\ell)\,
      \sqrt2\exp\(\tfrac34+O(n^{-b})\).
\end{equation*}
Note that $\sqrt2\, e^{3/4}\approx 2.9939$.

\section{Exact values}\label{S:exact}

As noted in Section~\ref{S:proof}, $M(n,\ell)$ is the number of
integer points in $\ell\sP_n$, where $\sP_n$ is the polytope defined
in that section.  Lattice point enumeration techniques such
as the algorithm in~\cite{LattE} therefore allow the exact 
computation of $M(n,\ell)$ for small $n$.
In practice this is feasible for $n\le 9$ or with difficulty
$n\le 10$, almost irrespective of~$\ell$.

By interpolating the computed values, we obtain the Ehrhart
quasi\-polynomial for small $n$.
Recall that $M(n,\ell)$ is a polynomial $M_e(n,\ell)$ for even~$\ell$
and a polynomial $M_o(n,\ell)$ for odd~$\ell$.
We have $M_o(n,\ell)=0$ if $n$ is odd, and the following.
\begin{align*}
        M_e(3,\ell) &= 1 \\[1ex]
        M_e(4,\ell) &= M_o(4,\ell) =
         \dfrac{1}{2}\,\ell^{2}
        + \dfrac{3}{2}\,\ell+1 \\[1ex]
        M_e(5,\ell) &= 
         \dfrac{5}{256}\,\ell^{5}
        + \dfrac{25}{128}\,\ell^{4}
        + \dfrac{155}{192}\,\ell^{3}
        + \dfrac{55}{32}\,\ell^{2}
        + \dfrac{47}{24}\,\ell+1 \displaybreak[0]\\[1ex]
        M_e(6,\ell) &=
         \dfrac{19}{120960}\,\ell^{9}
        + \dfrac{19}{5376}\,\ell^{8}
        + \dfrac{143}{4032}\,\ell^{7}
        + \dfrac{5}{24}\,\ell^{6}
        + \dfrac{4567}{5760}\,\ell^{5}
        + \dfrac{785}{384}\,\ell^{4} 
        + \dfrac{10919}{3024}\,\ell^{3} \\
        &{~}+ \dfrac{955}{224}\,\ell^{2}
        + \dfrac{857}{280}\,\ell+1 \displaybreak[0] \\
        M_o(6,\ell) &= M_e(6,\ell) 
        - \dfrac{5}{256} \displaybreak[0]\\[1ex]
        M_e(7,\ell) &=
         \dfrac{533}{3633315840}\,\ell^{14}
        + \dfrac{533}{86507520}\,\ell^{13}
        + \dfrac{279413}{2335703040}\,\ell^{12}
        + \dfrac{9233}{6488064}\,\ell^{11} \\
        &{~}+ \dfrac{3076459}{265420800}\,\ell^{10}
        + \dfrac{151339}{2211840}\,\ell^{9}
        + \dfrac{4679131}{15482880}\,\ell^{8}
        + \dfrac{9367}{9216}\,\ell^{7}
        + \dfrac{43502617}{16588800}\,\ell^{6} \\
        &{~}+ \dfrac{478009}{92160}\,\ell^{5}
        + \dfrac{71076539}{9123840}\,\ell^{4}
        + \dfrac{661673}{76032}\,\ell^{3}
        + \dfrac{1712147}{246400}\,\ell^{2}
        + \dfrac{9649}{2640}\,\ell+1 \displaybreak[0]\\[1ex]
        M_e(8,\ell) &= 
          \dfrac{70241}{5088422500761600}\,\ell^{20}
         + \dfrac{70241}{72691750010880}\,\ell^{19}
         + \dfrac{18703309}{585359881666560}\,\ell^{18} \\
        &{~}+ \dfrac{12330581}{18582853386240}\,\ell^{17}
         + \dfrac{428460619}{44144787456000}\,\ell^{16}
         + \dfrac{33009749}{310542336000}\,\ell^{15} \\
        &{~}+ \dfrac{90842880341}{100429391462400}\,\ell^{14}
         + \dfrac{5580172163}{910924185600}\,\ell^{13}
         + \dfrac{1110632463421}{33108590592000}\,\ell^{12} \\
        &{~}+ \dfrac{4381892419}{29196288000}\,\ell^{11}
         + \dfrac{304644862903}{551809843200}\,\ell^{10}
         + \dfrac{22001378209}{13138329600}\,\ell^{9} \\
        &{~}+ \dfrac{262880239845943}{62768369664000}\,\ell^{8}
         + \dfrac{12867890603299}{1494484992000}\,\ell^{7}
         + \dfrac {3890196991231}{269007298560}\,\ell^{6}\\
        &{~}+ \dfrac{9530810537}{485222400}\,\ell^{5}
         + \dfrac{76295531167}{3592512000}\,\ell^{4} 
         + \dfrac {1100694281}{61776000}\,\ell^{3}\\
         &{~}+ \dfrac{50787821048}{4583103525}\,\ell^{2}
         + \dfrac{135038369}{29099070}\,\ell+1 \displaybreak[0]\\
        M_o(8,\ell) &= M_e(8,\ell)
        - \dfrac{35}{1048576}\,\ell^{5}
        - \dfrac{1225}{2097152}\,\ell^{4}
        - \dfrac{13685}{3145728}\,\ell^{3}
        - \dfrac{17885}{1048576}\,\ell^{2}\\
        &{~}- \dfrac{233261}{6291456}\,\ell
        - \dfrac{78057}{2097152}  \displaybreak[0]\\[1ex]
   M_e(9,\ell) &= 
\dfrac{863924282670630091}{7732694804887618394297204736000000}\,\ell^{27}\\
&{~}+\dfrac{863924282670630091}{71599025971181651799048192000000}\,\ell^{26}\\
&{~}+\dfrac{10311705659720524879}{16522852147195765799780352000000}\,\ell^{25}
   \displaybreak[0]\\
&{~}+\dfrac{44159888290330963}{2145824954181268285685760000}\,\ell^{24}
+\dfrac{44603828594214317123}{91793623039976476665446400000}\,\ell^{23}\\
&{~}+\dfrac{4134171051301720697}{472621628924364010291200000}\,\ell^{22}
+\dfrac{2139768518991928638127}{17143275449165567282380800000}\,\ell^{21}\displaybreak[0]\\
&{~}+\dfrac{2365877475528196499}{1632692899920530217369600}\,\ell^{20}
+\dfrac{167364777037473990001}{12005094852356839833600000}\,\ell^{19}\displaybreak[0]\\
&{~}+\dfrac{43210221809651966023}{383621452048996761600000}\,\ell^{18}
+\dfrac{7598908879241416557943}{9846283935924250214400000}\,\ell^{17}
  \displaybreak[0]\\
&{~}+\dfrac{78473046995519797477}{17375795181042794496000}\,\ell^{16}
+\dfrac{2690417378247820105229333}{118589802110617072435200000}\,\ell^{15\displaybreak[0]}\\
&{~}+\dfrac{12598164604216578106061}{128343941678157004800000}\,\ell^{14}
+\dfrac{39802237244716247322233}{108598719881517465600000}\,\ell^{13}\displaybreak[0]\\
&{~}+\dfrac{183315648883655207683}{155141028402167808000}\,\ell^{12}
+\dfrac{11492891877126624163867}{3496549693154918400000}\,\ell^{11}\displaybreak[0]\\
&{~}+\dfrac{20646561932994651460327}{2622412269866188800000}\,\ell^{10}
+\dfrac{12699041960623534314853039}{784756871757456998400000}\,\ell^{9}\displaybreak[0]\\
&{~}+\dfrac{3536936635157608410019}{124564582818643968000}\,\ell^{8}
+\dfrac{602776622158017864239297}{14273025114636288000000}\,\ell^{7}\displaybreak[0]\\
&{~}+\dfrac{8959111748174759872739}{169916965650432000000}\,\ell^{6}
+\dfrac{62149609860286754066479}{1139859644571648000000}\,\ell^{5}\\
&{~}+\dfrac{416558573311485749}{9089789829120000}\,\ell^{4}
+\dfrac{7739053944610908107}{254233401117696000}\,\ell^{3}\\
&{~}+\dfrac{1309315468639693}{85753329742080}\,\ell^{2}
+\dfrac{94565099767}{17847429600}\,\ell+1
\end{align*}
The same method would yield $M(10,\ell)$ with a
plausible but large amount of computation.
For completeness, we also give the Ehrhart series
$L_n(x)=\sum_{\ell\ge 0} M(n,\ell)x^\ell$ for $n\le 9$.
\def\xx(#1,#2){(x^{#1}+\ifnum #2=0 1\else 
     \ifnum #2=1 x\else x^{#2}\fi\fi)}

\begin{align*}
(1-x^2)L_{{3}}(x)&=1\\[1ex]
(1-x)^{3}L_{{4}}(x)&=1\\[1ex]
(1-x^2)^{6}L_{{5}}(x)
&=\xx(8,0)+16\,\xx(6,2)+41\,{x}^{4}\displaybreak[0]\\[1ex]
(1-x)^{10}(1+x)L_{{6}}(x)
&=\xx(6,0)+6\,\xx(5,1)+30\,\xx(4,2)+40\,{x}^{3} \displaybreak[0]\\[1ex]
(1-x^2)^{15}L_{{7}}(x)
&=\xx(24,0)+807\,\xx(22,2)+81483\,\xx(20,4) \\
&{~}+1906342\,\xx(18,6)+15277449\,\xx(16,8) \\
&{~}+50349627
\,\xx(14,10)+74301542\,{x}^{12} \displaybreak[0]\\[1ex]
(1-x)^{21}(1+x)^{6}L_{{8}}(x)
&=\xx(20,0)+90\,\xx(19,1)+4726\,\xx(18,2) \\
&{~}+107050\,\xx(17,3)+1261121\,\xx(16,4) \\
&{~}+8761248\,\xx(15,5)+39187016\,\xx(14,6) \\
&{~}+119662536\,\xx(13,7)+259344246\,\xx(12,8) \\
&{~}+408811676\,\xx(11,9)+475095180\,{x}^{10} \displaybreak[0]\\[1ex]
(1-x^2)^{28}L_{{9}}(x)
&=\xx(48,0)+52524\,\xx(46,2)\\
&{~}+169345602\,\xx(44,4) \\
&{~}+78276428212\,\xx(42,6) \\
&{~}+10217460516057\,\xx(40,8) \displaybreak[0]\\
&{~}+527531262668208\,\xx(38,10) \displaybreak[0]\\
&{~}+13016462628712186\,\xx(36,12)\displaybreak[0] \\
&{~}+172410423955058664\,\xx(34,14) \displaybreak[0]\\
&{~}+1322251960254170931\,\xx(32,16) \\
&{~}+6176715510750440488\,\xx(30,18) \\
&{~}+18182086106689738044\,\xx(28,20) \\
&{~}+34470475812807166836\,\xx(26,22) \\
&{~}+42606701216240491693\,{x}^{24}
\end{align*}

For larger $n$, $\sP_n$ has too many vertices for this method to
be useful, but we can use the technique of~\cite{Mlabelled}
and~\cite{CMint}.  Define $f(z)=1+z+z^2+\cdots+z^\ell$.
Then $M(n,\ell)$ is the coefficient of 
$x_1^\ell x_2^\ell \cdots x_n^\ell y^{n\ell/2}$ in
$\prodjk f(x_jx_ky)$. 

If $q$ is any integer greater than
$\max\{n\ell/2, n^2(n-1)/2-n\ell/2\}$, then
$M(n,\ell)$ is the coefficient of the only term in
\begin{equation*} 
	y^{-n\ell/2} x_1\cdots x_n \prodjk f(x_jx_ky) 
\end{equation*}
in which each $x_j$ appears with a power that is a multiple of~$\ell+1$
and $y$ appears with a power that is a multiple of~$q$.
Now let $p$ be a prime number for which $p-1$ is a multiple of
both $\ell+1$ and~$q$.
Let $\alpha$ and $\beta$ be a primitive $(\ell+1)$-th root
and a primitive $q$-th root of unity in $\Z_p$, respectively.
Then, modulo~$p$,
\begin{align*}
 M(n,\ell) &= \frac{n!}{q(\ell+1)^n}\\
  &{\hskip-1em}\times \sum_{r_0+\cdots+r_d=n}\,
   \prod_{i=0}^d \frac{\alpha^{ir_i}}{r_i!}
     \sum_{k=0}^	{q-1} \beta^{-kn\ell/2}
     \prod_{i=0}^{d} f(\alpha^{2i}\beta^k)^{\binom{r_i}{2}}
     \!\!\prod_{0\le i<j\le d} \! f(\alpha^{i+j}\beta^k)^{r_ir_j},
\end{align*}
where the first summation is over all non-negative integers
$r_0,r_1,\ldots,r_d$ which sum to~$n$.  Using sufficiently many
primes~$p$, we can extract the exact value of $M(n,\ell)$ using
the Chinese Remainder Theorem.
As an example of a value computed using this method, we have
{\small
\[
	M(19,10)
   = 613329062511931789477677176839174642138032757885191693120,
\]}
which is about 2\%\ higher than the estimate of Theorem~\ref{Th:main}.

Machine-readable versions of these exact formulas, along with many
other exact values of $M(n,\ell)$, can be found
at~\cite{Table}.

After observing a large number of exact values, we have noted that~\eqref{E:binform} appears to have an accuracy much wider than we can prove. We can even guess extra terms.
We express our observations in the following conjecture.

\begin{conjecture}\label{Conj:bold}
For even $n\ell$, define $\varDelta(n,\ell)$ by
\begin{equation*}
   M(n,\ell) =  M_{\mathrm{naive}}(n,\ell)\,
      \sqrt2\exp\biggl( \frac34 + \frac{3\ell+1}{12\ell(n-1)}
                                 + \frac{\varDelta(n,\ell)}{n(n-1)}\biggr).
\end{equation*}
Then $\abs{\varDelta(n,\ell)} < 1$ for $n\ge 5$, $\ell\ge 1$.
\end{conjecture}

\nicebreak
\section{The minimum entry}\label{S:simple}

In this section we note a simple corollary of Theorem~\ref{Th:main}.
Choose~$X$ uniformly at random from the set
$\M(n,\ell)$ of zero-diagonal symmetric non-negative
integer matrices of order $n$ and row sums~$\ell$.
Let $X_{\mathrm{min}}$ be the least off-diagonal entry of $X$.  If
$X_{\mathrm{min}}\ge k$ for integer $k\ge 0$,
we can subtract~$k$
from each entry to make a matrix of row sums~$\ell-(n-1)k$.
This elementary observation shows that
\[
   \Prob(X_{\mathrm{min}}\ge k) = \frac{M(n,\ell-(n-1)k)}{M(n,\ell)}.
\]
Theorem~\ref{Th:main} can thus be used to estimate this probability
whenever it applies to the quantities on the right.  We can provide some information
even in other cases; note that~\eqref{E:constant} is not required
for the following.

\begin{thm}\label{Th:min}
  Let $k=k(n)\ge 0$ and $\ell=\ell(n)\ge 0$ with $n\ell$ even.
  Define $a=kn^3/\ell$.
  Then, as $n\to\infty$,
  \[
     \Prob(X_{\mathrm{min}}\ge k)\;
     \begin{cases}
           {}\to 0 & \text{ if }a\to\infty\\
           {}\sim e^{-a/2} & \text{ if } a = O(1).
     \end{cases}
  \]
\end{thm}
\begin{proof}
  We begin with a case incompletely covered by Theorem~\ref{Th:main},
  namely $\ell=o(n^3)$.  Define $\M_0,\M_1$ to be the sets of those matrices
  in $\M(n,\ell)$ with no off-diagonal zeros, and exactly
  two or four off-diagonal zeros, respectively.
  Given $X\in\M_0$, choose distinct $q,r,s,t$  
  and replace
  $a_{qr}, a_{rs}, a_{st}, a_{tq}$ (and $a_{rq}, a_{sr}, a_{ts}, a_{qt}$
  consistently) by $a_{qr}-\delta, a_{rs}+\delta,
  a_{st}-\delta, a_{tq}+\delta$, where $\delta=\min\{a_{qr},a_{st}\}$.
  This can be done in $\varTheta(n^4)$ ways and creates an 
  element of~$\M_1$.
  Alternatively, if $X\in\M_1$, choose distinct $q,r,s,t$ such that
  either $a_{qr}$ or $a_{st}$ or both are 0.  Then replace
  $a_{qr}, a_{rs}, a_{st}, a_{tq}$ (and $a_{rq}, a_{sr}, a_{ts}, a_{qt}$
  consistently) by $a_{qr}+\delta, a_{rs}-\delta,
  a_{st}+\delta, a_{tq}-\delta$, where $1\le\delta\le\min\{a_{rs},a_{tq}\}-1$.
  If this produces an element of $\M_0$, it is the inverse of the
  previous operation.
  Given a choice of $a_{qr}=0$, $s$ and $\delta$ can be chosen in at
  most $\ell$ ways since $\sum_s a_{rs}=\ell$, then $t$ can be chosen
  in at most $n$ ways.  Similarly for $q_{st}=0$.
  Therefore, this operation can be done in at most $O(\ell n)$ ways.
  It follows that either $\abs{\M_0}=0$ or $\abs{\M_0}=o(\abs{\M_1})$,
  which completes this case since
  $\Prob(X_{\mathrm{min}}\ge k) \le \Prob(X_{\mathrm{min}}\ge 1)$
  for $k\ge 1$.

  In case $\ell=\varTheta(n^3)$, define $k'=\min\{k,\lfloor \ell/(2n)\rfloor\}$
  and estimate the value of $\Prob(X_{\mathrm{min}}\ge k')$ using \eqref{E:biglambda}.
  This gives the desired result when $k=k'$.  For $k>k'$ the value obtained
  tends to~0, so again the desired result follows by monotonicity with
  respect to~$k$.
\end{proof}

\nicebreak
\section{Concluding remarks}\label{S:conclusion}

In this paper we have begun the asymptotic enumeration of dense
symmetric non-negative integer matrices with given row sums, by
considering the special case of uniform row sums and zero diagonal.
Further cases, which can be approached by the same method, are
to allow the row sums to vary, and to allow diagonals other than
zero.  The structure of random matrices in the class can also be
investigated by specifying some forced matrix entries.
We hope to return to these problems in the future.

\nicebreak
\bibliographystyle{line}

\end{document}